\documentclass[12pt,reqno]{article}

\usepackage[usenames]{color}
\usepackage{amssymb}
\usepackage{amsmath}
\usepackage{amsthm}
\usepackage{amsfonts}
\usepackage{amscd}
\usepackage{graphicx}

\usepackage[colorlinks=true,
linkcolor=webgreen,
filecolor=webbrown,
citecolor=webgreen]{hyperref}

\definecolor{webgreen}{rgb}{0,.5,0}
\definecolor{webbrown}{rgb}{.6,0,0}

\usepackage{color}
\usepackage{fullpage}
\usepackage{float}

\usepackage{graphics}
\usepackage{latexsym}
\usepackage{epsf}
\usepackage{breakurl}

\newcommand{\seqnum}[1]{\href{https://oeis.org/#1}{\rm \underline{#1}}}

\begin{document}


\theoremstyle{plain}
\newtheorem{theorem}{Theorem}
\newtheorem{corollary}[theorem]{Corollary}
\newtheorem{lemma}[theorem]{Lemma}
\newtheorem{proposition}[theorem]{Proposition}

\theoremstyle{definition}
\newtheorem{definition}[theorem]{Definition}
\newtheorem{example}[theorem]{Example}
\newtheorem{conjecture}[theorem]{Conjecture}

\theoremstyle{remark}
\newtheorem{remark}[theorem]{Remark}

\begin{center}
\vskip 1cm
{\LARGE\bf Arithmetic Identities for Some Analogs of $5$-core Partition Function}

\vskip 1cm
\large
Subhajit Bandyopadhyay\\
Department of Mathematical Sciences, Tezpur University\\
Napaam 784028, Assam\\
India\\
\href{mailto:subhajitbn.maths@gmail.com}{\tt subhajitbn.maths@gmail.com} \\

\vspace{.4cm}
Nayandeep Deka Baruah\\
Department of Mathematical Sciences, Tezpur University\\
Napaam 784028, Assam\\
India\\
\href{mailto:nayan@tezu.ernet.in}{\tt nayan@tezu.ernet.in} \\
\end{center}

\vskip .2 in
\begin{abstract}
Recently, Gireesh, Ray, and Shivashankar studied an analog, $\overline{a}_t(n)$, of the $t$-core partition function, $c_t(n)$. In this paper, we study the function $\overline{a}_5(n)$ (\seqnum{A053723}) in conjunction with $c_5(n)$ (\seqnum{A368490}) as well as another analogous function $\overline{b}_5(n)$ (\seqnum{A368495}). We also find several arithmetic identities for $\overline{a}_5(n)$ and $\overline{b}_5(n)$.
\end{abstract}

\section{Introduction}
A \emph{partition} $\lambda=(\lambda_1,\lambda_2,\ldots,\lambda_k)$
of a natural number $n$ is a finite sequence of non-increasing
positive integer \emph{parts} $\lambda_i$ such that
$n=\sum_{i=1}^k\lambda_i$. The \emph{Ferrers--Young diagram} of the
partition $\lambda$ of $n$ is constructed by placing $n$ nodes in $k$
rows so that the $i$th row has $\lambda_i$ nodes. The nodes are marked with the row and column coordinates, similar to how one would mark the position of the elements of a matrix. Let $\lambda_j^\prime $ denote the number of nodes in column $j$. The \emph{hook number} $H(i,j)$ for the node at position $(i,j)$ is determined by counting the nodes situated directly below and to the right of it, including the node itself. That is, $H(i,j)=\lambda_i+\lambda_j^\prime-i-j+1$. If none of the hook numbers of a partition is divisible by $t$, then it is called a \emph{$t$-core}.

\begin{example} The Ferrers--Young diagram of the
partition $\lambda=(4,3,1,1)$ of $9$ is
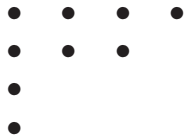
\begin{center}\begin{tabular}{p{.3cm}p{.3cm}p{.3cm}p{.3cm}}
{$\bullet$} & {$\bullet$} & {$\bullet$} & {$\bullet$}\\
{$\bullet$}&{$\bullet$}&{$\bullet$}\\
{$\bullet$}\\{$\bullet$}
\end{tabular}
\end{center}
The nodes $(1,1)$, $(1,2)$, $(1,3)$, $(1,4)$, $(2,1)$, $(2,2)$, $(2,3)$, $(3,1)$, and
$(4,1)$ have hook numbers 7, 4, 3, 1, 5, 2, 1, 2, and 1, respectively.
Therefore, $\lambda$ is a $t$-core for $t=6$ and $t\ge 8$.
\end{example}

Granville and Ono \cite{go} proved that for $t\ge4$, every natural
number $n$ has a $t$-core. For a recent survey on $t$-cores, we refer the readers to a paper by Cho, Kim, Nam, and Sohn \cite{HBHJ}.

If $c_t(n)$ denotes the number of $t$-cores of $n$, then its generating function is given by (see \cite[Eq.\ (2.1)]{GKS}) 
\begin{equation}\label{ct}
	\sum_{n=0}^{\infty} c_t(n)q^n = \dfrac{f_t^t}{f_1},
\end{equation}
where for integer $j\geq1$, $f_j:=(q^j;q^j)_\infty$ and throughout the paper, for complex numbers $a$ and $q$ with $|q|<1$, we define
$$(a;q)_\infty:=\prod_{k=0}^\infty (1-aq^k).$$

For $|ab|<1$, Ramanujan's general theta function $f(a,b)$ is defined by
\begin{align*}
f(a,b):=\displaystyle\sum_{n=-\infty}^{\infty}a^{{n(n+1)}/2}b^{{n(n-1)}/2}.
\end{align*}
In this notation, Jacobi's well-known triple product identity \cite[p.\ 35, Entry 19]{BCB3} takes the form
\begin{align}
\label{JTPI} f(a,b)=(-a;ab)_{\infty}(-b;ab)_{\infty}(ab;ab)_{\infty}.
\end{align}

Consider the following special cases of $f(a,b)$: 
\begin{align}\label{phi-q}\varphi(-q)&:=f(-q,-q)=\sum_{n=-\infty}^\infty (-1)^nq^{n^2}=\dfrac{f_1^2}{f_2},\\
\label{psi}
\psi(-q)&:=f(-q,-q^3)=\sum_{n=0}^\infty (-q)^{n(n+1)/2}=\dfrac{f_1f_4}{f_2},\\\label{f-q}
f(-q)&:=f(-q,-q^2)=\sum_{n=-\infty}^\infty (-1)^n q^{n(3n-1)/2}=f_1,
\end{align}
where the $q$-product representations in the above arise from \eqref{JTPI} and manipulation of the $q$-products.

In the notation of \eqref{f-q}, the generating function \eqref{ct} of $c_t(n)$ may be recast as
\begin{equation}\label{ctf}
	\sum_{n=0}^{\infty} c_t(n)q^n = \dfrac{f^t(-q^t)}{f(-q)}.
\end{equation}
 
Recently, Gireesh, Ray, and Shivashankar \cite[Eq.\ (1.2)]{GRS} considered an analog $\overline{a}_t(n)$ of $c_t(n)$ with $f(-q)$ is replaced by $\varphi(-q)$ in \eqref{ctf}, namely, 
\begin{equation*}
	\sum_{n=0}^{\infty} \overline{a}_t(n)q^n = \dfrac{\varphi^t(-q^t)}{\varphi(-q)}.
\end{equation*}
They obtained some arithmetic identities and multiplicative formulas for $\overline{a}_3(n)$, $\overline{a}_4(n)$, and $\overline{a}_8(n)$ by using Ramanujan's theta functions (It is to be noted that Theorem 1.1 in their paper \cite{GRS} holds only for a special case. The induction process in the proof of the theorem is not quite correct). Employing the theory of modular forms, they also studied the arithmetic density of $\overline{a}_t(n)$ and found the following Ramanujan type congruence for $\overline{a}_5(n)$ \cite[Theorem 1.10]{GRS}: For all $n\geq0$,
\begin{align}\label{girees_a5_congruence}
	\overline{a}_5(20n+6) \equiv 0 \pmod{5}.
\end{align}

Note that
\begin{align}\label{at}
	\sum_{n=0}^{\infty} \overline{a}_5(n)q^n& = \dfrac{\varphi^5(-q^5)}{\varphi(-q)} \notag\\
	&= 1+2q+4q^2+8q^3+14q^4+14q^5+20q^6+24q^7+\cdots.
\end{align}

In this paper, we revisit the function $\overline{a}_5(n)$ in conjunction with $c_5(n)$ as well as another function $\overline{b}_5(n)$ defined by
\begin{equation}\label{bt}
	\sum_{n=0}^{\infty} \overline{b}_5(n)q^n = \dfrac{\psi^5(-q^5)}{\psi(-q)} = 1+q+q^2+2q^3+3q^4-q^5+2q^7-2q^9+6q^{10}+\cdots,
\end{equation}
where $\psi(-q)$ is defined in \eqref{psi}.

The sequences $(c_5(n))$, $(\overline{a}_5(n))$, and $(\overline{b}_5(n))$ are \seqnum{A053723}, \seqnum{A368490}, and \seqnum{A368495}, respectively, in \cite{OEIS}.

We state our results in the following theorems and corollaries. In the sequel, we assume that $c_5(n)=\overline{a}_5(n)=\overline{b}_5(n)=0$ for $n<0$. 

A recurrence relation for $\overline{a}_5(n)$ and some relations between $\overline{a}_5(n)$ and $c_5(n)$ are stated in the following theorem. 
\begin{theorem}\label{a5_main_theorem}
	For every nonnegative integer $n$,
	\begin{align}
		\overline{a}_5(5n+2) &= 4c_5(5n+1), \label{a5n2_thm_statement}\\
		\overline{a}_5(5n+3) &= 4c_5(5n+2), \label{a5n3_thm_statement} \\
		\overline{a}_5(10n+1) &= 2c_5(10n), \label{a10n1_thm_statement}\\
		\overline{a}_5(10n+9) &= 2c_5(10n+8), \label{a10n9_thm_statement} \\
		\overline{a}_5(20n+6) &= 10c_5(10n+2),\label{a5_main_theorem_1}\\
		\overline{a}_5(20n+14) &= 10c_5(10n+6).\label{a5_main_theorem_2}
	\end{align}
	Furthermore, for every integer $k\geq2$,
	\begin{align}\label{a5_main_theorem_3}
		\overline{a}_5(5^{k}n) = \left(\dfrac{5^{k} - 1}{4}\right)\overline{a}_5(5n) - \left(\dfrac{5^{k} - 5}{4}\right)\overline{a}_5(n).
	\end{align}
\end{theorem}

The following corollary is immediate from the above theorem.
\begin{corollary} For every nonnegative integer $n$ and every integer $k\geq2$,
	\begin{align}\label{girees_a5_congruence-mod10a}
		\overline{a}_5(20n+6) &\equiv 0 \pmod{10},\\
		\label{girees_a5_congruence-mod10b}\overline{a}_5(20n+14) &\equiv 0 \pmod{10},\\\intertext{and}
		4\overline{a}_5(5^{k}n) &\equiv 5\overline{a}_5(n) - \overline{a}_5(5n) \pmod{5^k}.\notag
	\end{align}
\end{corollary}
Note that \eqref{girees_a5_congruence-mod10a} implies \eqref{girees_a5_congruence}. However, even stronger results implying \eqref{girees_a5_congruence-mod10a} and \eqref{girees_a5_congruence-mod10b} are stated in Corollary \ref{cor-gireesh}.

Now we state some recurrence relations for $\overline{b}_5(n)$.
\begin{theorem}\label{b5_main_theorem1}
	For every nonnegative integer $n$ and every integer $k\geq2$, we have
	\begin{align}\label{b5_main_theorem_preliminary}
		\overline{b}_5(4n+3) &= 2\overline{b}_5(2n)\\\intertext{and}
		\label{b5_main_theorem_5}
		\overline{b}_5\left(5^{k}(n+3)-3\right) &= \left(\dfrac{5^{k}-1}{4}\right)\overline{b}_5(5n+12) - \left(\dfrac{5^{k}-5}{4}\right)\overline{b}_5(n).
	\end{align}
\end{theorem}
Next we state some identities connecting $\overline{b}_5(n)$ with $\overline{a}_5(n)$ and $c_5(n)$.
\begin{theorem}\label{b5_main_theorem}
	For every nonnegative integer $n$, we have
	\begin{align}\label{b5_main_theorem_preliminary1}
		\overline{b}_5(4n+1) &= c_5(n)-2\overline{b}_5(2n-1),\\
		\overline{b}_5(10n) &= \dfrac{1}{2}c_5(10n+2) ,\label{b5_10n}\\
		\overline{b}_5(10n+1) &= c_5(5n+1),\label{b5_10n_1}\\
		\overline{b}_5(10n+2) &= \dfrac{1}{4}\overline{a}_5(2n+1) +\dfrac{1}{2}c_5(2n), \label{b5_10n_2}\\
		\overline{b}_5(10n+3) &= c_5(5n+2),\label{b5_10n_3}\\
		\overline{b}_5(10n+4) &= \dfrac{1}{2}c_5(10n+6), \label{b5_10n_4}\\
		\overline{b}_5(10n+6) &= 0, \label{b5_10n_6}\\
		\overline{b}_5(10n+8) &= 0, \label{b5_10n_8}\\
		\overline{b}_5(20n+5) &= -c_5(5n+1), \label{b5_20n_5}\\
		\overline{b}_5(20n+7) &= \dfrac{1}{2}\overline{a}_5(2n+1) +c_5(2n), \label{b5_20n_7}\\
		\overline{b}_5(20n+9) &= -c_5(5n+2), \label{b5_20n_9}\\
		\overline{b}_5(20n+15) &= 0, \label{b5_20n_15}\\ 
		\overline{b}_5(20n+19) &= 0. \label{b5_20n_19}
	\end{align}
\end{theorem}
\begin{corollary} For positive integers $n$, $\overline{b}_5(n)$ is $0$ for at least $30\%$, greater than $0$ for at least $52\%$, and less than $0$ for at least $10\%$.
\end{corollary}
\begin{proof}
	Identities \eqref{b5_10n_6}, \eqref{b5_10n_8}, \eqref{b5_20n_15}, and \eqref{b5_20n_19} readily imply the observed frequency of zeroes. Similarly, \eqref{b5_20n_5} and \eqref{b5_20n_9} imply the frequency of negatives. From the identities of \eqref{a5n2_thm_statement}, \eqref{a5n3_thm_statement}, \eqref{a10n1_thm_statement}, and \eqref{a10n9_thm_statement}, we observe that the sequence $(\overline{a}_5(2n+1))$ is positive in at least 4 out of 5 cases. Together with \eqref{b5_10n}--\eqref{b5_10n_4} and \eqref{b5_20n_7}, this implies that the frequency of positives is at least equal to 
$$\dfrac{2+2+2\times(4/5)+2+2+1\times(4/5)}{20},$$ that is, 52\%.
\end{proof}

From \eqref{a5_main_theorem_1}, \eqref{a5_main_theorem_2}, \eqref{b5_10n}, and \eqref{b5_10n_4} we arrive at the following corollary, implying the congruence of \eqref{girees_a5_congruence} by Gireesh, Ray, and Shivashankar \cite[Thm.\ 1.10]{GRS}.
\begin{corollary}\label{cor-gireesh}
	For $n$ being any non-negative integer,
	\begin{align}
		\overline{a}_5(20n+6) = 20\overline{b}_5(10n), \\ 
		\overline{a}_5(20n+14) = 20\overline{b}_5(10n+4).
	\end{align}
\end{corollary}
We sate some infinite families of congruences in the following corollary. 
\begin{corollary}
	For every nonnegative integer $n$ and every integer $k\geq2$,
	\begin{align*}
		4\overline{b}_5\left(5^{k}(n+3)-3\right) &\equiv 5\overline{b}_5(n) - \overline{b}_5(5n+12) \pmod{5^k},\\\overline{b}_5(5^k(20n+18)-3) &\equiv 0 \pmod{\dfrac{5^{k}-1}{4}},\\\intertext{and}
		\overline{b}_5(5^k(20n+22)-3) &\equiv 0 \pmod{\dfrac{5^{k}-1}{4}}.		
	\end{align*}
\end{corollary}
\begin{proof}
	The first congruence readily follows from \eqref{b5_main_theorem_5}. Again, from \eqref{b5_20n_15}, \eqref{b5_20n_19} and \eqref{b5_main_theorem_5} it follows that, for every nonnegative integer $n$ and every integer $k\geq2$,
	\begin{align*}
		\overline{b}_5(5^k(20n+18)-3) &= \left(\dfrac{5^{k}-1}{4}\right)\overline{b}_5(100n+87),\\
		\overline{b}_5(5^k(20n+22)-3) &= \left(\dfrac{5^{k}-1}{4}\right)\overline{b}_5(100n+107),
	\end{align*}
	which implies the last two congruences in the corollary.
\end{proof}

We arrange the rest of the paper as follows. In Section \ref{analogues-sec2}, we provide some preliminary lemmas. Section \ref{analogues-sec3} is devoted to proving the identities stated in Theorem \ref{a5_main_theorem}. The proofs of Theorem \ref{b5_main_theorem1} and Theorem \ref{b5_main_theorem} are given in Section \ref{analogues-sec4} and Section \ref{analogues-sec5}, respectively. 

\section{Preliminary Lemmas}\label{analogues-sec2}

In the following lemma, we state some known theta function identities.
\begin{lemma}If $\varphi(-q)$, $\psi(-q)$, and $f(-q)$ are as defined in \eqref{phi-q}--\eqref{f-q} and $\chi(-q):=(q;q^2)_\infty$, then 
	\begin{align}
		\label{phimodeq}\dfrac{\varphi^5(q^5)}{\varphi(q)} + 4q\dfrac{f^5(q^5)}{f(q)} &= \varphi(q) \varphi^3(q^5),\\
		\label{phimodeqfora5}\varphi^2(q) - \varphi^2(q^5) &= 4q\chi(q)f_{5}f_{20},\\
		\label{psimodeq}\dfrac{\psi^5(-q^5)}{\psi(-q)} - \dfrac{\psi^5(q^5)}{\psi(q)} &= 4q^3\dfrac{\psi^5(q^{10})}{\psi(q^2)} + 2q\dfrac{f_{20}^5}{f_4},\\
		\label{psimodeqforb5}\psi^2(q) - q\psi^2(q^5) &= \dfrac{f(-q^5) \varphi(-q^5)}{\chi(-q)}=f(q,q^4)f(q^2,q^3),\\
		\label{f5modeg}\dfrac{f_5^5}{f_1} - 4q^3\dfrac{f_{20}^5}{f_4} &= \dfrac{f^5(q^5)}{f(q)} + 2q\dfrac{f_{10}^5}{f_2},\\
		\label{A+4B}\dfrac{f_2^2}{f_1^4}&=\dfrac{f_{10}^2}{f_5^4}+4q\dfrac{f_2f_{10}^5}{f_1^3f_5^5}.
	\end{align}
\end{lemma}
\begin{proof} Identities \eqref{phimodeq} and \eqref{phimodeqfora5} are identical to Entry 9(ii) and Entry 9(iii) of \cite[Chap.\ 19]{BCB3}. For the proofs of \eqref{psimodeq} and \eqref{psimodeqforb5}, we refer to Entry 15 and Entry 18 of \cite[Chap.\ 36]{BCB5}. The identity \eqref{f5modeg} can be found in \cite[Eq.\ (4.7)]{NDB-BCB}. Identity \eqref{A+4B} is simply \cite[Eq.\ (2.6)]{NDB-NMB}.
\end{proof}

In the following lemma, we recall some 5-dissection formulas from \cite[p.\ 85, Eq.\ (8.1.1) and p.\ 89, Eq.\ (8.4.4)]{H_QBook} and \cite[p.\ 49, Corollary]{BCB3}.
\begin{lemma}
	Let $R(q)$ be defined as
	\begin{equation*}
		R(q) := \dfrac{(q;q^5)_{\infty}(q^4;q^5)_{\infty}}{(q^2;q^5)_{\infty}(q^3;q^5)_{\infty}}.
	\end{equation*}
	We have
	\begin{align}
		f_1 &= \dfrac{1}{R(q^{5})} -q -q^2 R(q^{5}),\label{f1_5dissection}\\
		\dfrac{1}{f_1} &= \dfrac{f_{25}^5}{f_{5}^6}\Big(R(q^{5})^{-4} + qR(q^{5})^{-3} + 2q^{2}R(q^{5})^{-2} + 3q^{3}R(q^{5})^{-1} + 5q^{4} \nonumber\\
		&\quad - 3q^{5}R(q^{5}) + 2q^{6}R(q^{5})^{2} - q^{7}R(q^{5})^{3} + q^{8}R(q^{5})^{4} \Big),\label{partition_5dissection}\\
		\varphi(q) &= \varphi(q^{25}) + 2qf(q^{15}, q^{35}) + 2q^{4}f(q^{5}, q^{45}),\label{phi_5dissection}\\
		\psi(q) &= f(q^{10}, q^{15}) + qf(q^5, q^{20}) + q^3\psi(q^{25}) \label{psi_5dissection}.
	\end{align}	
\end{lemma}

In the following lemma, we present two useful identities on $c_5(n)$. 
\begin{lemma}
	For every nonnegative integer $n$,
	\begin{align}
		c_5(4n+1) &= c_5(2n), \label{c4n1}\\
		c_5(5n+4) &= 5c_5(n). \label{c5n4}
	\end{align}	
\end{lemma}
\begin{proof}
	See \cite[Eq.\ (4.8)]{NDB-BCB} and \cite[Eq.\ (5.1)]{GKS}.
\end{proof}

\section{Proof of Theorem \ref{a5_main_theorem}}\label{analogues-sec3}

\vspace{.4cm}\noindent \emph{Proofs of \eqref{a5n2_thm_statement} and \eqref{a5n3_thm_statement}}.
Replacing $q$ by $-q$ in \eqref{phimodeq}, we have
\begin{align}\label{starting_point_for_a5_rec_proof}
	\dfrac{\varphi^5(-q^5)}{\varphi(-q)} = 4q\dfrac{f_5^5}{f_1} + \varphi(-q) \varphi^3(-q^5),
\end{align}
which, by \eqref{ct} and \eqref{at}, may be recast as 
\begin{equation}
	\sum_{n=0}^{\infty}\overline{a}_5(n)q^n = 4\sum_{n=0}^{\infty}c_5(n)q^{n+1} + \varphi(-q)\varphi^3(-q^5).
\end{equation}

Replacing $q$ by $-q$ in \eqref{phi_5dissection} and then using the resulting identity in the above, we have 
\begin{align}\label{starting_point_for_a5_rec_proof-2}
	\sum_{n=0}^{\infty}\overline{a}_5(n)q^n &= 4\sum_{n=0}^{\infty}c_5(n)q^{n+1} \nonumber\\
	&\quad+ \varphi^3(-q^5)\left(\varphi(-q^{25}) -2qf(-q^{15}, -q^{35}) +2q^4f(-q^{5}, -q^{45}) \right).
\end{align}
Equating the coefficients of $q^{5n+2}$ and $q^{5n+3}$ from both sides of the above, we arrive at \eqref{a5n2_thm_statement} and \eqref{a5n3_thm_statement}, respectively.

\vspace{.4cm}\noindent \emph{Proofs of \eqref{a10n1_thm_statement} and \eqref{a10n9_thm_statement}}.
Multiplying both sides of \eqref{f5modeg} by $\dfrac{f_5^5 f_2}{f_{10}^5 f_1}$, we have
\begin{equation*}
	\dfrac{\varphi^5(-q^5)}{\varphi(-q)} - 4q^3\dfrac{\psi^5(-q^5)}{\psi(-q)} = \dfrac{\varphi^5(-q^{10})}{\varphi(-q^2)} + 2q\dfrac{f_5^5}{f_1}.
\end{equation*}
which, by \eqref{ct}, \eqref{at}, and \eqref{bt}, yields
\begin{equation}\label{abc-i}
	\sum_{n=0}^{\infty} \overline{a}_5(n)q^n - 4\sum_{n=0}^{\infty}\overline{b}_5(n)q^{n+3} = \sum_{n=0}^{\infty} \overline{a}_5(n)q^{2n} + 2\sum_{n=0}^{\infty}c_5(n)q^{n+1}.
\end{equation}
Comparing the coefficients of $q^{2n+1}$ from both sides, we find that 
\begin{align}
	\overline{a}_5(2n+1) - 4\overline{b}_5(2n-2) &= 2c_5(2n).\label{a2n1}
\end{align}
Replacing $n$ by $5n$ and $5n+4$, we obtain
\begin{align*}
	\overline{a}_5(10n+1) = 4\overline{b}_5(10n-2) + 2c_5(10n)\\\intertext{and}
	\overline{a}_5(10n+9) = 4\overline{b}_5(10n+6) + 2c_5(10n+8),
\end{align*}
respectively. Using \eqref{b5_10n_6} and \eqref{b5_10n_8} in the above, we arrive at \eqref{a10n1_thm_statement} and \eqref{a10n9_thm_statement}.

\vspace{.4cm}\noindent \emph{Proofs of \eqref{a5_main_theorem_1} and \eqref{a5_main_theorem_2}}. Equating the coefficients of $q^{2n}$ from both sides of \eqref{abc-i}, we have
\begin{align}
	\overline{a}_5(2n) - 4\overline{b}_5(2n-3) &= \overline{a}_5(n) + 2c_5(2n-1)\label{a2n}.
\end{align}
From \eqref{a2n1} and \eqref{a2n}, it follows that
\begin{align}
	\overline{a}_5(4n+2) - 4\overline{b}_5(4n-1) &= \overline{a}_5(2n+1) + 2c_5(4n+1),\label{a4n2}\\
	\overline{a}_5(4n) - 4\overline{b}_5(4n-3) &= \overline{a}_5(2n) + 2c_5(4n-1),\label{a4n}\\
	\overline{a}_5(4n+1) - 4\overline{b}_5(4n-2) &= 2c_5(4n),\label{a4n1}\\
	\overline{a}_5(4n+3) - 4\overline{b}_5(4n) &= 2c_5(4n+2).\label{a4n3}
\end{align}

Again, employing \eqref{ct} and \eqref{bt}, it follows from \eqref{psimodeq} that
\begin{equation}\label{b4n-i}
	\sum_{n=0}^{\infty} \overline{b}_5(n)q^n - \sum_{n=0}^{\infty} \overline{b}_5(n)(-q)^n=4\sum_{n=0}^{\infty}(-1)^n\overline{b}_5(n)q^{2n+3}+2 \sum_{n=0}^{\infty} c_5(n)q^{4n+1}.
\end{equation}
Equating the coefficients of $q^{4n+3}$ from both sides of the above, we have
\begin{align}
	\overline{b}_5(4n+3) &= 2\overline{b}_5(2n).\label{b4n3} 
\end{align}
It follows from \eqref{a2n1} and \eqref{b4n3} that
\begin{align*}
	\overline{a}_5(2n+1) = 2c_5(2n) + 2\overline{b}_5(4n-1).
\end{align*}
Using \eqref{c4n1} and the above identity in \eqref{a4n2}, we obtain
\begin{align}\label{a4n2_v2}
	\overline{a}_5(4n+2) = 3\overline{a}_5(2n+1) -2c_5(2n),
\end{align}
which by replacement of $n$ with $5n+1$ yields 
\begin{align}\label{a4n2_v2i}
	\overline{a}_5(20n+6) = 3\overline{a}_5(10n+3) - 2c_5(10n+2).
\end{align}

Again, replacing $n$ by $2n$ in \eqref{a5n3_thm_statement}, we have
\begin{align}
	\overline{a}_5(10n+3) &= 4c_5(10n+2). \label{name4}
\end{align}
It follows from \eqref{a4n2_v2i} and \eqref{name4} that
\begin{align*}
	\overline{a}_5(20n+6) = 10c_5(10n+2), 
\end{align*}
which is \eqref{a5_main_theorem_1}.

Next, replacing $n$ by $5n+3$ in \eqref{a4n2_v2}, we have
\begin{align} \label{a5-20n14}
	\overline{a}_5(20n+14) = 3\overline{a}_5(10n+7) - 2c_5\left(10n+6\right).
\end{align}

Again, replacing $n$ by $2n+1$ in \eqref{a5n2_thm_statement}, we have
\begin{align}
	\overline{a}_5(10n+7) &= 4c_5(10n+6). \label{name3}
\end{align} 
It follows from \eqref{a5-20n14} and \eqref{name3} that
\begin{align*}
	\overline{a}_5(20n+14) = 10c_5(10n+6),
\end{align*}
which is \eqref{a5_main_theorem_2}.

\vspace{.4cm}\noindent \emph{Proof of \eqref{a5_main_theorem_3}}.
With the aid of \eqref{at}, we recast \eqref{starting_point_for_a5_rec_proof} as
\begin{align}\label{a5_rec_main_1}
	\sum_{n=0}^{\infty}\overline{a}_5(n)q^n = 4q\dfrac{f_5^5}{f_1} + \varphi(-q) \varphi^3(-q^5).
\end{align}
Employing the 5-dissections of $\varphi(-q)$ from \eqref{phi_5dissection} and that of $1/f_1$ from \eqref{partition_5dissection} in the above identity, we have
\begin{align*}
	\sum_{n=0}^{\infty}\overline{a}_5(n)q^n &= 4q\dfrac{f_{25}^5}{f_5}\Big(R(q^{5})^{-4} + qR(q^{5})^{-3} + 2q^{2}R(q^{5})^{-2} + 3q^{3}R(q^{5})^{-1} + 5q^{4} \\
		&\quad - 3q^{5}R(q^{5}) + 2q^{6}R(q^{5})^{2} - q^{7}R(q^{5})^{3} + q^{8}R(q^{5})^{4} \Big) \\
		&\quad+ \varphi^3(-q^5)\Big(\varphi(-q^{25}) - 2qf(-q^{15}, -q^{35}) + 2q^{4}f(-q^{5}, -q^{45})\Big).
\end{align*}
Extracting the terms involving $q^{5n}$ from both sides of the above, and then replacing $q^5$ by $q$, we find that
\begin{align}\label{a5_rec_main_2}
	\sum_{n=0}^{\infty}\overline{a}_5(5n)q^n = 20q\dfrac{f_5^5}{f_1} + \varphi^3(-q)\varphi(-q^5).
\end{align}
Subtracting \eqref{a5_rec_main_1} from \eqref{a5_rec_main_2},
\begin{align}\label{a5_eliminate_from_this_1}
	\sum_{n=0}^{\infty}\overline{a}_5(5n)q^n - \sum_{n=0}^{\infty}\overline{a}_5(n)q^n &= 16q\dfrac{f_5^5}{f_1} + \varphi(-q)\varphi(-q^5)\left(\varphi^2(-q) - \varphi^2(-q^5)\right).
\end{align}
Employing \eqref{partition_5dissection} and \eqref{phi_5dissection} in the above and then extracting the terms involving $q^{5n}$, we obtain
\begin{align*}
	&\sum_{n=0}^{\infty}\overline{a}_5(25n)q^n - \sum_{n=0}^{\infty}\overline{a}_5(5n)q^n \\
	&= 80q\dfrac{f_5^5}{f_1} + \varphi(-q)\left(\varphi^3(-q^5) - 24q\varphi(-q^5)f(-q^{3},-q^{7})f(-q,-q^{9})\right)\\
	&\quad - \varphi(-q)\varphi(-q^5).
\end{align*}
Replacing $q$ by $-q$ in \eqref{phimodeqfora5} and then employing in the above identity, we find that
\begin{align*}
	\sum_{n=0}^{\infty}\overline{a}_5(25n)q^n - \sum_{n=0}^{\infty}\overline{a}_5(5n)q^n &= 80q\dfrac{f_5^5}{f_1} + 5\varphi(-q)\varphi(-q^5)\left(\varphi^2(-q) - \varphi^2(-q^5)\right),
\end{align*}
which, by \eqref{a5_eliminate_from_this_1}, yields
\begin{equation*}
	\sum_{n=0}^{\infty}\overline{a}_5(25n)q^n - \sum_{n=0}^{\infty}\overline{a}_5(5n)q^n = 5\sum_{n=0}^{\infty}\overline{a}_5(5n)q^n - 5\sum_{n=0}^{\infty}\overline{a}_5(n)q^n.
\end{equation*}
Equating the coefficients of $q^n$ from both sides, we find that, for any nonnegative integer $n$,
\begin{align}
	\overline{a}_5(25n) = 6\overline{a}_5(5n) - 5\overline{a}_5(n).
\end{align}
Now \eqref{a5_main_theorem_3} follows by mathematical induction on $k\geq2$.

\section{Proof of Theorem \ref{b5_main_theorem1}} \label{analogues-sec4}

Note that \eqref{b5_main_theorem_preliminary} is identical to \eqref{b4n3}. Therefore, we proceed to prove only \eqref{b5_main_theorem_5}.

Replacing $q$ by $-q$ in \eqref{psimodeqforb5}, we have
\begin{align*}
	q\psi^2(-q^5) = \dfrac{f(q^5) \varphi(q^5)}{\chi(q)}-\psi^2(-q).
\end{align*}
Multiplying both sides of the above identity by $\dfrac{\psi^3(-q^5)}{\psi(-q)}$, we find that
\begin{align}
	q\dfrac{\psi^5(-q^5)}{\psi(-q)} &= \dfrac{f_{10}^5}{f_2} - \psi(-q)\psi^3(-q^5), \label{b5_rec_main_2_proof_starting_point} 
\end{align}
which, by \eqref{bt}, can be recast as
\begin{align}\label{b5-rec-proof-new1}
	\sum_{n=0}^{\infty}\overline{b}_5(n)q^{n+1} &= \dfrac{f_{10}^5}{f_2}  -  \psi(-q)\psi^3(-q^5).
\end{align}
Employing the 5-dissection of $\psi(-q)$ from \eqref{psi_5dissection} and that of $1/f_2$ from \eqref{partition_5dissection} in \eqref{b5_rec_main_2_proof_starting_point}, and then extracting the terms involving $q^{5n+3}$ from both sides of the resulting identity, we obtain 
\begin{align}
	\sum_{n=0}^{\infty}\overline{b}_5(5n+2)q^{n} &= 5q\dfrac{f_{10}^{5}}{f_2} + \psi^3(-q)\psi(-q^5). \label{b5_rec_main_2}
\end{align}
Multiplying \eqref{b5-rec-proof-new1} by $q$ and subtracting from \eqref{b5_rec_main_2},
\begin{align}\label{b5_eliminate_from_this_1}
	&\sum_{n=0}^{\infty}\overline{b}_5(5n+2)q^{n} -\sum_{n=0}^{\infty}\overline{b}_5(n)q^{n+2} \nonumber\\
	&= 4q\dfrac{f_{10}^{5}}{f_2} + \psi(-q)\psi(-q^5)\left(\psi^2(-q) + q\psi^2(-q^5)\right).
\end{align}
Again, using \eqref{partition_5dissection} and \eqref{psi_5dissection} in the above identity and extracting the terms involving $q^{5n+4}$ from both sides, we have
\begin{align}\label{b5_rec_proof_before_last_step}
	&\sum_{n=0}^{\infty}\overline{b}_5(25n+22)q^{n} -\sum_{n=0}^{\infty}\overline{b}_5(5n+2)q^{n} \nonumber\\
	&= 20q\dfrac{f_{10}^{5}}{f_2} + \psi(-q)\left(6\psi(-q^5)f(q^{2},-q^{3})f(-q,q^{4}) - q\psi^3(-q^5)\right) \nonumber\\
	&\quad - \psi^3(-q)\psi(-q^5).
\end{align}
Replacing $q$ by $-q$ in \eqref{psimodeqforb5} and employing in the above identity, we obtain
\begin{align}\label{b5_eliminate_from_this_2}
	&\sum_{n=0}^{\infty}\overline{b}_5(25n+22)q^{n} - \sum_{n=0}^{\infty}\overline{b}_5(5n+2)q^{n} \nonumber\\ 
	&= 20q\dfrac{f_{10}^{5}}{f_2} + 5\psi(-q)\psi(-q^5)\left(\psi^2(-q) + q\psi^2(-q^5)\right).
\end{align}
From \eqref{b5_eliminate_from_this_1} and \eqref{b5_eliminate_from_this_2}, it follows that
\begin{align*}
	&\sum_{n=0}^{\infty}\overline{b}_5(25n+22)q^{n} - \sum_{n=0}^{\infty}\overline{b}_5(5n+2)q^{n} = 5\sum_{n=0}^{\infty}\overline{b}_5(5n+2)q^{n}  - 5\sum_{n=0}^{\infty}\overline{b}_5(n)q^{n+2}.
\end{align*}
Comparing the coefficients of $q^{n}$ from both sides of the above equation, we find that, for any nonnegative integer $n$,
\begin{equation}
	\overline{b}_5(25n+72) = 6\overline{b}_5(5n+12) - 5\overline{b}_5(n).
\end{equation}
The general recurrence relation \eqref{b5_main_theorem_5} now follows by mathematical induction on $k\geq2$.

\section{Proof of Theorem \ref{b5_main_theorem}}\label{analogues-sec5}

\vspace{.4cm}\noindent\emph{Proofs of \eqref{b5_main_theorem_preliminary1}, \eqref{b5_10n}, and \eqref{b5_10n_4}}. Equating the coefficients of $q^{4n+1}$ from both sides of \eqref{b4n-i}, have
\begin{align*}
	\overline{b}_5(4n+1) &= c_5(n) -2\overline{b}_5(2n-1),	
\end{align*} which is \eqref{b5_main_theorem_preliminary1}.

Replacing $n$ by $n+1$ in \eqref{a2n1} and rearranging the terms,
\begin{equation}\label{starting_for_b5_10n_manipulations}
	4\overline{b}_5(2n) = \overline{a}_5(2n+3) - 2c_5(2n+2).
\end{equation}
Replacing $n$ by $5n$ in the above identity and using \eqref{name4}, we have
\begin{align*}
	4\overline{b}_5(10n) &= \overline{a}_5(10n+3) - 2c_5(10n+2) \\
	&= 4c_5(10n+2) - 2c_5(10n+2) \\
	&= 2c_5(10n+2),
\end{align*}
which leads to \eqref{b5_10n}. 

Next, replacing $n$ by $5n+2$ in \eqref{starting_for_b5_10n_manipulations} and employing \eqref{name3}, we obtain
\begin{align*}
	4\overline{b}_5(10n+4) &= \overline{a}_5(10n+7) - 2c_5(10n+6) \\
	&= 4c_5(10n+6) - 2c_5(10n+6) \\
	&= 2c_5(10n+6),
\end{align*}
implying \eqref{b5_10n_4}.

\vskip .4cm
\noindent \emph{Proofs of \eqref{b5_10n_1}, \eqref{b5_10n_3}, \eqref{b5_10n_6}, and \eqref{b5_10n_8}}. Employing \eqref{psi_5dissection} in \eqref{b5-rec-proof-new1}, we have
\begin{align}
	\label{b5_rec_main_1}&\sum_{n=0}^{\infty}\overline{b}_5(n)q^{n+1}\notag\\
	&= \sum_{n=0}^{\infty}c_5(n)q^{2n} - \psi^3(-q^5)\left(f(q^{10}, -q^{15}) -qf(-q^5, q^{20}) -q^3\psi(-q^{25})\right).
\end{align}
Comparing the coefficients of the terms involving $q^{10n+2}$, $q^{10n+4}$, $q^{10n+7}$, and $q^{10n+9}$ from both sides of the above identity, we arrive at the desired results of \eqref{b5_10n_1}, \eqref{b5_10n_3}, \eqref{b5_10n_6}, and \eqref{b5_10n_8}, respectively.

\vspace{.4cm}\noindent \emph{Proofs of \eqref{b5_20n_5}, \eqref{b5_20n_9}, \eqref{b5_20n_15}, and \eqref{b5_20n_19}}. Replacing $n$ by $5n+1$ in \eqref{b5_main_theorem_preliminary1} and then applying \eqref{b5_10n_1},
\begin{align*}
	\overline{b}_5(20n+5) &= c_5(5n+1) - 2\overline{b}_5(10n+1) \\
	&= c_5(5n+1) - 2c_5(5n+1) \\
	&= -c_5(5n+1),
\end{align*}
which proves \eqref{b5_20n_5}. 

Similarly, replacing $n$ by $5n+2$ in \eqref{b5_main_theorem_preliminary1} and using \eqref{b5_10n_3}, we arrive at \eqref{b5_20n_9}.

Replacing $n$ by $5n+3$ in \eqref{b5_main_theorem_preliminary} and then employing \eqref{b5_10n_6}, we have
\begin{align*}
	\overline{b}_5(20n+15) &= 2\overline{b}_5(10n+6) \\
	&= 0,
\end{align*}
which proves \eqref{b5_20n_15}. 

In a similar manner, replacing $n$ by $5n+4$ in \eqref{b5_main_theorem_preliminary} and utilizing \eqref{b5_10n_8}, we obtain \eqref{b5_20n_19}. 

\vskip .4cm
\noindent \emph{Proofs of \eqref{b5_10n_2} and \eqref{b5_20n_7}}.
From \eqref{phi-q} and \eqref{A+4B}, we see that
\begin{align*}
	\varphi^3(-q)\varphi(-q^5) &= \dfrac{f_1^6 f_5^2}{f_2^3 f_{10}} \\
	&= \dfrac{f_1^2 f_5^6}{f_2 f_{10}^3} -4q\dfrac{f_1^3 f_5 f_{10}^2}{f_2^2} \\
	&= \varphi(-q)\varphi^3(-q^5) -4q\dfrac{f_5^5}{f_1} +16q^2\dfrac{f_{10}^5}{f_2}.
\end{align*}
Utilizing \eqref{a5_eliminate_from_this_1}, the above identity can be recast as
\begin{equation*}
	\sum_{n=0}^{\infty}\overline{a}_5(5n)q^n = \sum_{n=0}^{\infty}\overline{a}_5(n)q^n +12q\dfrac{f_5^5}{f_1} +16q^2\dfrac{f_{10}^5}{f_2}.
\end{equation*}
Extracting the terms with odd powers of $q$ from both sides, we arrive at
\begin{equation}\label{a5_10n_5}
	\overline{a}_5(10n+5) = \overline{a}_5(2n+1) +12c_5(2n).
\end{equation}

Now, replacing $n$ by $10n+5$ in \eqref{a2n},
\begin{align*}
	4\overline{b}_5(20n+7) &= \overline{a}_5(20n+10) - \overline{a}_5(10n+5) - 2c_5(20n+9).
\end{align*}
Employing \eqref{a4n2_v2} with $n$ replaced by $5n+2$ and \eqref{c4n1} with $n$ replaced by $5n+2$ in the above identity, and then using \eqref{c5n4}, we find that
\begin{align*}	
	4\overline{b}_5(20n+7)&= 3\overline{a}_5(10n+5) - 2c_5(10n+4)  - \overline{a}_5(10n+5) - 2c_5(10n+4) \\
	&= 2\overline{a}_5(10n+5)-4c_5(10n+4) \\
	&= 2\overline{a}_5(10n+5)-20c_5(2n).
\end{align*}
Applying \eqref{a5_10n_5} in the above expression, we obtain
\begin{equation*}
	2\overline{b}_5(20n+7) = \overline{a}_5(2n+1) +2c_5(2n)
\end{equation*}
which implies \eqref{b5_20n_7}.

Finally, replacing $n$ by $5n+1$ in \eqref{b5_main_theorem_preliminary} and then applying \eqref{b5_20n_7}, we have
\begin{align*}
	\overline{b}_5(10n+2) &= \dfrac{1}{2}\overline{b}_5(20n+7) \\
	&= \dfrac{1}{4}\overline{a}_5(2n+1) +\dfrac{1}{2}c_5(2n),
\end{align*}
which is \eqref{b5_10n_2}.

\section{Acknowledgment}
The first author was partially supported by an INSPIRE Fellowship for Doctoral Research, DST, Government of India. The author acknowledges the funding agency.

\end{document}